\documentclass[a4paper,12pt]{article}

\usepackage[left=2cm,right=2cm, top=2cm,bottom=3cm,bindingoffset=0cm]{geometry}

\usepackage{verbatim}
\usepackage{amsmath}
\usepackage{amsthm}
\usepackage{amssymb}
\usepackage{delarray}
\usepackage{cite}
\usepackage{hyperref}
\usepackage{mathrsfs}
\usepackage{tikz}
\usetikzlibrary{patterns}
\usepackage{caption}
\DeclareCaptionLabelSeparator{dot}{. }
\newcommand{\ga}{\gamma}
\newcommand{\de}{\delta}
\newcommand{\la}{\lambda}
\newcommand{\om}{\omega}

\newcommand{\eps}{\varepsilon}

\newcommand{\iy}{\infty}

\theoremstyle{plain}

\numberwithin{equation}{section}

\newtheorem{thm}{Theorem}[section]
\newtheorem{lem}[thm]{Lemma}
\newtheorem{prop}[thm]{Proposition}

\theoremstyle{definition}

\newtheorem{alg}[thm]{Algorithm}
\newtheorem{ip}[thm]{Inverse Problem}

\theoremstyle{remark}

\DeclareMathOperator*{\Res}{Res}

 \allowdisplaybreaks

\begin{document}

\begin{center}
{\large\bf Inverse Sturm-Liouville problem with analytical functions in the boundary condition}
\\[0.2cm]
{\bf Natalia P. Bondarenko} \\[0.2cm]
\end{center}

\vspace{0.5cm}

{\bf Abstract.} The inverse spectral problem is studied for the Sturm-Liouville operator with a complex-valued potential and arbitrary entire functions in one of the boundary conditions. We obtain necessary and sufficient conditions for uniqueness, and develop a constructive algorithm for the inverse problem solution. The main results are applied to the Hochstadt-Lieberman half-inverse problem. As an auxiliary proposition, we prove local solvability and stability for the inverse Sturm-Liouville problem by the Cauchy data in the non-self-adjoint case.
  
\medskip


\medskip


\vspace{1cm}

\section{Introduction}

The paper aims to solve the inverse spectral problem for the following boundary value problem
\begin{gather} \label{eqv}
    -y''(x) + q(x) y(x) = \la y(x), \quad x \in (0, \pi), \\ \label{bc}
    y(0) = 0, \quad f_1(\la) y'(\pi) + f_2(\la) y(\pi) = 0.
\end{gather}

Here \eqref{eqv} is the Sturm-Liouville equation with the complex-valued potential $q \in L_2(0, \pi)$. The boundary condition at $x = \pi$ contains arbitrary functions $f_j(\la)$, $j = 1, 2$, analytical by the spectral parameter $\la$ in the whole complex plane.
The Sturm-Liouville equation~\eqref{eqv} arises in investigation of wave propagation in various media, heating processes, electron motion, etc.

The case of constant coefficients $f_1$ and $f_2$ has been studied fairly completely (see the classical monographs \cite{Mar77, Lev84, PT87, FY01} and references therein). There is also a number of studies concerning inverse problems for Sturm-Liouville operators with linear \cite{BS97, Yur00, Gul05, YH10} and polynomial \cite{Chug01, FY10, FY12, YX15, Gul19} dependence on the spectral parameter in the boundary conditions.

In this paper, we study the Sturm-Liouville problem with arbitrary entire functions in the boundary condition. Let $\{ \la_n \}_{n = 1}^{\iy}$ be a subsequence of the eigenvalues of the problem $L(q)$. This subsequence may coincide with the whole spectrum or not. Note that the behavior of the spectrum depends very much on the functions $f_j(\la)$, $j = 1, 2$. 
Since no additional restrictions are imposed on these functions, we cannot investigate certain properties of the spectrum. Nevertheless, we can study the following inverse problem under some additional restrictions on the subspectrum $\{ \la_n \}_{n = 1}^{\iy}$.

\begin{ip} \label{ip:main}
Let the entire functions $f_j(\la)$, $j = 1, 2$, be known a priori. Given the eigenvalues $\{ \la_n \}_{n = 1}^{\iy}$ and the number $\om := \frac{1}{2}\int_0^{\pi} q(x) \, dx$, find the potential $q$.
\end{ip}

Investigation of this problem is motivated by several applications. In recent years, the so-called {\it partial} inverse problems have attracted much attention of scholars. In such problems, it is assumed that coefficients of differential expressions (e.g., the Sturm-Liouville potential $q(x)$) are known a priori on a part of an interval. Therefore less spectral data are required to recover the unknown part of coefficients. A significant part of those partial inverse problems can be reduced to Inverse Problem~\ref{ip:main} for the operator with analytical dependence on the spectral parameter in the boundary conditions. We provide an example of such reduction for the Hochstadt-Lieberman problem~\cite{HL78} in Section~5. Recently partial inverse problems have been intensively studied for Sturm-Liouville operators with discontinuities (see \cite{Hald84, SY08, Yang14, Wang15, YB19}). The latter operators arise in geophysics and electronics. Partial inverse problems have also been investigated for differential operators on geometrical graphs (see \cite{Piv00, Yang10, YW17, Bond18, Bond-tamk}). Such operators model wave propagation through a domain being a thin neighborhood of a graph and have applications in various branches of science and engineering (see \cite{BCFK06}). Another popular problem is the inverse transmission eigenvalue problem arising in acoustics (see \cite{MP94, MPS94, MSS97, BB17}). The results of the present paper generalize many known results on the mentioned inverse problems. Note that, in certain applications, the constant $\om$ can be obtained from the eigenvalue asymptotics (e.g., see Section~4).

In this paper, we obtain necessary and sufficient conditions for uniqueness of Inverse Problem~\ref{ip:main} solution and develop a constructive algorithm for solving this inverse problem. This algorithm will be used in our future study \cite{Bond-future} for investigation of solvability and stability for Inverse Problem~\ref{ip:main}. Further this theory can be generalized to other types of differential operators and pencils.

Our method is based on completeness and basisness of special vector-functional sequences in appropriate Hilbert spaces. This method allows us to reduce Inverse Problem~\ref{ip:main} to the classical Sturm-Liouville inverse problem with constant coefficients in the boundary conditions. In contrast to the majority of the studies on inverse Sturm-Liouville problems, our analysis does not require self-adjointness of the operator. 
We investigate the most general case, when the potential $q(x)$ is complex-valued and the given eigenvalues can be multiple. For solving the inverse Sturm-Liouville problem with boundary conditions independent of the spectral parameter, we rely 
on the inverse problem theory for non-self-adjoint Sturm-Liouville operators developed in \cite{FY01, But07, BSY13}.

The paper is organized as follows. In Section~2, we introduce the notations, and formulate the main results, in particular, necessary and sufficient conditions for uniqueness of solution (Theorems~\ref{thm:uniq} and~\ref{thm:nec}) and Algorithm~\ref{alg:ip} for constructive solution of the inverse problem. 
We also provide some simple conditions on the subspectrum $\{ \la_n \}_{n = 1}^{\iy}$ sufficient for uniqueness and for constructive solution (see Theorem~\ref{thm:cond}).
The main theorems are proved in Section~3. In Section~4, we apply our results to the Hochstadt-Lieberman problem. 
In Appendix, Theorem~\ref{thm:loc} on local solvability and stability is proved for the inverse Sturm-Liouville problem by Cauchy data. This result plays an auxiliary role in analysis of Inverse Problem~\ref{ip:main}. However, as far as we know, Theorem~\ref{thm:loc} is new for the case of the complex-valued potential $q(x)$ and so can be treated as a separate result.

\section{Main results}

Let us start with some preliminaries.
Denote by $S(x, \la)$ the solution of equation~\eqref{eqv}, satisfying the initial conditions $S(0, \la) = 0$, $S'(0, \la) = 1$. Here and below the prime stands for the derivative by $x$.
For derivatives by $\la$, we use the following notation: 
$$
f^{<j>}(\la) = \frac{1}{j!}\frac{d^j}{d \la^j} f(\la), \quad j \ge 0.
$$

The spectrum of $L(q)$ consists of eigenvalues, which coincide with the zeros the the characteristic function
\begin{equation} \label{defD}
\Delta(\la) := f_1(\la) S'(\pi, \la) + f_2(\la) S(\pi, \la).
\end{equation}
Clearly, the function $\Delta(\la)$ is entire in $\la$-plane.

Consider a subsequence $\{ \la_n \}_{n = 1}^{\iy}$ of the spectrum. Any multiple eigenvalue can appear in the sequence $\{ \la_n \}_{n = 1}^{\iy}$ a number of times not exceeding its multiplicity. By the eigenvalue multiplicity we mean the multiplicity of the corresponding zero of the analytic function $\Delta(\la)$. In other words, if for some $\mu$ we have $\# \{ n \in \mathbb N \colon \la_n = \mu \} = k$, then
$\Delta^{<j>}(\mu) = 0$, $j = \overline{0, k-1}$. We call such a sequence $\{ \la_n \}_{n = 1}^{\iy}$ \textit{a subspectrum} of $L(q)$.

Let us add to the the given subspectrum the value $\la_0 := 0$. Define
\begin{equation} \label{defI}
I := \{ n \ge 0 \colon \la_n \ne \la_k, \, \forall k \colon 0 \le k < n \}, \quad
m_n := \# \{ k \ge 0 \colon \la_k = \la_n \},
\end{equation}
i.e. $I$ is the index set of all the distinct values among $\{ \la_n \}_{n = 0}^{\iy}$ and $m_n$ is the multiplicity of $\la_n$ for $n \in I$. Without loss of generality, we assume that the equal eigenvalues are consecutive: $\la_n = \la_{n + 1} = \dots = \la_{n + m_n - 1}$ for all $n \in I$.

Define the functions
$$
s(x,\la) = \sqrt{\la} \sin (\sqrt{\la} x), \quad c(x, \la) = \cos (\sqrt{\la} x).
$$
Obviously, the functions $\la^{-1} s(x, \la)$ and $c(x, \la)$ are entire by $\la$ for each fixed $x \in [0, \pi]$. Define $\eta_1(\la) := S(\pi, \la)$, $\eta_2(\la) := S'(\pi, \la)$. Further we need the following standard relations, which can be obtained by using the transformation operator (see \cite{Mar77, FY01, MP10}):
\begin{align} \label{intS}
    \eta_1(\la) = \frac{s(\pi, \la)}{\la} - \frac{\om c(\pi, \la)}{\la} + \frac{1}{\la} \int_0^{\pi} K(t) c(t, \la) \, dt, \\ \label{intSp}
    \eta_2(\la)  = c(\pi, \la) + \frac{\om s(\pi, \la)}{\la} + \frac{1}{\la} \int_0^{\pi} N(t) s(t, \la) \, dt,
\end{align}
where $K, N \in L_2(0, \pi)$. The pair of functions $\{ K, N \}$ is called \textit{the Cauchy data} of the potential $q$. Consider the following auxiliary inverse problem.

\begin{ip} \label{ip:cd}
Given the Cauchy data $\{ K, N \}$ and the number $\om$, find the potential $q$.
\end{ip}

Using the Cauchy data $\{ K, N \}$ and $\om$, one can easily construct the Weyl function $M(\la) := \frac{\eta_2(\la)}{\eta_1(\la)}$. It is well-known that the potential $q$ can be uniquely recovered from the Weyl function, e.g., by the method of spectral mappings (see \cite{FY01, But07, BSY13}).

Proceed to solution of Inverse Problem~\ref{ip:main}. Substituting \eqref{intS} and~\eqref{intSp} into~\eqref{defD}, we 
get
\begin{align} \nonumber
\la \Delta(\la) & = f_1(\la) \left(\la c(\pi, \la) + \om s(\pi, \la) + \int_0^{\pi} N(t) s(t, \la) \, dt\right) \\ \label{rel1} & + f_2(\la) \left( s(\pi, \la) - \om c(\pi, \la) + \int_0^{\pi} K(t) c(t, \la) \, dt\right).
\end{align}

Introduce the complex Hilbert space of vector-functions:
$$
\mathcal H := L_2(0, \pi) \oplus L_2(0, \pi) = \left\{ h = [h_1, h_2] \colon h_j \in L_2(0, \pi), \, j = 1, 2 \right\}
$$
with the following scalar product and the norm:
\begin{gather*}
(g, h)_{\mathcal H} := \int_0^{\pi} (\overline{g_1(t)} h_1(t) + \overline{g_2(t)} h_2(t)) \, dt, \quad
\| h \|_{\mathcal H} = \sqrt{(h, h)_{\mathcal H}}, \\ g, h \in \mathcal H, \quad g = [g_1, g_2], \quad h = [h_1, h_2].
\end{gather*}
Define the vector-functions
\begin{equation} \label{defv}
u(t) := [\overline{N(t)}, \overline{K(t)}], \quad v(t, \la) := [f_1(\la) s(t, \la), f_2(\la) c(t, \la)].
\end{equation}
Clearly, $u(.)$ and $v^{<\nu>}(., \la)$ for each fixed $\la$ and $\nu \ge 0$ belong to $\mathcal H$. In view of our notations, the relation~\eqref{rel1} can be rewritten in the form
\begin{gather} \nonumber
(u(t), v(t, \la))_{\mathcal H} = \la \Delta(\la) + w(\la), \\ \label{defw}
w(\la) := -f_1(\la) (\la c(\pi, \la) + \om s(\pi, \la)) - f_2(\la) (s(\pi, \la) - \om c(\pi, \la)).
\end{gather}
Here $t$ is the variable of integration in the scalar product.
Since
\begin{equation} \label{derD}
(\la \Delta(\la))^{<\nu>}_{|\la = \la_n} = 0, \quad n \in I, \quad \nu = \overline{0, m_n-1}, 
\end{equation}
we get
\begin{equation} \label{rel2}
    (u(t), v^{<\nu>}(t, \la_n))_{\mathcal H} = w^{<\nu>}(\la_n), \quad n \in I, \quad \nu = \overline{0, m_n-1}.
\end{equation}
Denote
\begin{gather} \label{defvn1}
v_{n+\nu}(t) := v^{<\nu>}(t, \la_n), \quad w_{n + \nu} := w^{<\nu>}(\la_n), \quad n \in I, \quad \nu = \overline{0, m_n-1}, \quad n + \nu \ge 1, \\ \label{defvn2}
v_0(t) := [0, 1], \quad w_0 := \om.
\end{gather}
Finally, we get
\begin{equation}\label{scal}
    (u, v_n)_{\mathcal H} = w_n, \quad n \ge 0.
\end{equation}
The relation~\eqref{scal} for $n \ge 1$ follows from~\eqref{rel2}. For $n = 0$, \eqref{scal} follows from~\eqref{intS}, since $S(\pi, \la)$ is analytical at $\la = 0$. In view of~\eqref{defv}, \eqref{defw}, \eqref{defvn1} and~\eqref{defvn2}, the vector-functions $\{ v_n \}_{n = 0}^{\iy}$ and the numbers $\{ w_n \}_{n = 0}^{\iy}$ can be constructed by the given data of Inverse Problem~\ref{ip:main}. The components of $u$ can help to find the unknown potential $q$.

Introduce the following conditions.

\medskip

\textsc{(Complete)} The sequence $\{ v_n \}_{n = 0}^{\iy}$ is complete in $\mathcal H$.

\smallskip

\textsc{(Basis)} The sequence $\{ v_n \}_{n = 0}^{\iy}$ is an unconditional basis in $\mathcal H$.

\medskip

Indeed, \textsc{(Basis)} implies \textsc{(Complete)}.

Along with the problem $L(q)$, we consider the problem $L(\tilde q)$ of the form~\eqref{eqv}-\eqref{bc} with another potential $\tilde q \in L_2(0, \pi)$.
The functions $f_j(\la)$, $j = 1, 2$, are the same for these two problems.
We agree that, if a certain symbol $\ga$ denotes an object related to $L(q)$, the symbol $\tilde \ga$ with tilde denotes the analogous object related to $L(\tilde q)$. Now we are ready to formulate the uniqueness theorem for Inverse Problem~\ref{ip:main}.

\begin{thm} \label{thm:uniq}
Let $\{ \la_n \}_{n = 1}^{\iy}$ and $\{ \tilde \la_n \}_{n = 1}^{\iy}$ be subspectra of the problems $L(q)$ and $L(\tilde q)$, respectively. Suppose that $L(q)$ and $\{ \la_n \}_{n = 0}^{\iy}$ satisfy the condition \textsc{(Complete)}, and let $\la_n = \tilde \la_n$, $n \ge 1$, $\om = \tilde \om$. Then $q = \tilde q$ in $L_2(0, \pi)$.
\end{thm}

The following theorem asserts that the condition \textsc{(Complete)} is not only sufficient but also necessary for uniqueness of solution of Inverse Problem~\ref{ip:main}.

\begin{thm} \label{thm:nec}
Let $\{ \la_n \}_{n = 1}^{\iy}$ be a subspectrum of the problem $L(q)$. Suppose that the sequence $\{ v_n \}_{n = 0}^{\iy}$ is incomplete in $\mathcal H$. Then there exists a complex-valued function $\tilde q \in L_2(0, \pi)$, $\tilde q \ne q$ such that $\om = \tilde \om$ and $\{ \la_n \}_{n = 1}^{\iy}$ is a subspectrum of $L(\tilde q)$.
\end{thm}

Suppose that the condition \textsc{(Basis)} holds. Then one can constructively solve Inverse Problem~\ref{ip:main}, by using the following algorithm.

\begin{alg} \label{alg:ip}
Let eigenvalues $\{ \la_n \}_{n = 1}^{\iy}$ and the number $\om$ be given. We need to find the potential $q$.
\begin{enumerate}
    \item Using $f_j(\la)$, $j = 1, 2$, $\{ \la_n \}_{n = 0}^{\iy}$ and $\om$, construct the vector-functions $\{ v_n \}_{n = 0}^{\iy}$ and the numbers $\{ w_n \}_{n = 0}^{\iy}$ via~\eqref{defv}, \eqref{defw}, \eqref{defvn1} and~\eqref{defvn2}.
    \item For the basis $\{ v_n \}_{n = 0}^{\iy}$, find the biorthonormal basis $\{ v_n^* \}_{n = 0}^{\iy}$, i.e. $(v_n, v_k^*)_{\mathcal H} = \de_{nk}$, $n, k \ge 0$, where $\de_{nk}$ is the Kronecker delta.
    \item Construct the element $u \in \mathcal H$, satisfying~\eqref{scal}, by the formula
    $$
    u = \sum_{n = 0}^{\iy} \overline{w_n} v_n^*.
    $$
    \item Using the components of $u(t) = [\overline{N(t)}, \overline{K(t)}]$, solve Inverse Problem~\ref{ip:cd} and find $q$.
\end{enumerate}
\end{alg}

In certain applications, it can be difficult to check the conditions~\textsc{(Complete)} and~\textsc{(Basis)}. Therefore we introduce some other conditions, sufficient for uniqueness and for constructive solution of Inverse Problem~\ref{ip:main}.

\medskip

\textsc{(Complete2)} The sequence $\{ c^{<\nu>}(t, \la_n) \}_{n \in I, \, \nu = \overline{0, m_n-1}}$ is complete in $L_2(0, 2\pi)$.

\smallskip

\textsc{(Basis2)} The sequence $\{ c^{<\nu>}(t, \la_n) \}_{n \in I, \, \nu = \overline{0, m_n-1}}$ is a Riesz basis in $L_2(0, 2\pi)$.

\smallskip

\textsc{(Separation)} For every $n \ge 0$, we have $f_1(\la_n) \ne 0$ or $f_2(\la_n) \ne 0$.

\smallskip

\textsc{(Simple)} There exists an integer $n_0$ such that $m_n = 1$ and $\la_n \ne 0$ for $n \ge n_0$.

\smallskip

\textsc{(Asymptotics)} $\mbox{Im}\, \rho_n = O(1)$, $n \to \iy$, and $\{ \rho_n^{-1} \}_{n \ge n_0} \in l_2$, where $\rho_n := \sqrt{\la_n}$, $\arg \rho_n \in \left[ -\tfrac{\pi}{2}, \tfrac{\pi}{2}\right)$.

\medskip

These conditions are natural for applications, such as the Hochstadt-Lieberman problem (see Section~5), transmission inverse eigenvalue problem, inverse problems for quantum graphs, etc.

\begin{thm} \label{thm:cond}
(i) \textsc{(Separation)} and \textsc{(Complete2)}  together imply \textsc{(Complete)};
(ii) \textsc{(Separation)}, \textsc{(Simple)}, \textsc{(Asymptotics)} and~\textsc{(Basis2)} together imply~\textsc{(Basis)}.
\end{thm}

Thus, one can change the condition \textsc{(Complete)} in Theorem~\ref{thm:uniq} to \textsc{(Separation)} and \textsc{(Complete2)} and the condition \textsc{(Basis)} in Algorithm~\ref{alg:ip} to \textsc{(Separation)}, \textsc{(Simple)}, \textsc{(Asymptotics)} and \textsc{(Basis2)}. Those results remain valid.

The condition \textsc{(Separation)} is essential for investigation of Inverse Problem~\ref{ip:main}. If this condition is violated, i.e. $f_1(\la_n) = f_2(\la_n) = 0$ for some $n$, in view of~\eqref{defD}, the eigenvalue $\la_n$ carries no information on the potential $q$. It is easy to check, that \textsc{(Separation)} follows from~\textsc{(Complete)}, so \textsc{(Separation)} is implicitly required in the uniqueness Theorem~\ref{thm:uniq} and in Algorithm~\ref{alg:ip}.

\section{Proofs}

The aim of this section is to prove Theorems~\ref{thm:uniq}, \ref{thm:nec} and~\ref{thm:cond}.

\begin{proof}[Proof of Theorem~\ref{thm:uniq}]
Suppose that the problems $L(q)$, $L(\tilde q)$ and their subspectra $\{ \la_n \}_{n = 1}^{\iy}$, $\{ \tilde \la_n \}_{n = 1}^{\iy}$ satisfy the conditions of Theorem~\ref{thm:uniq}. 
By virtue of the definitions~\eqref{defv}, \eqref{defw}, \eqref{defvn1} and~\eqref{defvn2}, we have $v_n = \tilde v_n$ in $\mathcal H$ and $w_n = \tilde w_n$ for all $n \ge 0$. Hence the relation~\eqref{scal} for $\tilde L$ has the form
\begin{equation} \label{scal2}
    (\tilde u, v_n)_{\mathcal H} = w_n, \quad n \ge 0.
\end{equation}
Subtracting \eqref{scal2} from~\eqref{scal} and using the completeness of the sequence $\{ v_n \}_{n = 0}^{\iy}$, we get $u = \tilde u$ in $\mathcal H$, i.e. $K = \tilde K$, $N = \tilde N$  in $L_2(0, \pi)$. Using the uniqueness of Inverse Problem~\ref{ip:cd} solution, we conclude that $q = \tilde q$ in $L_2(0, \pi)$.
\end{proof}

\begin{proof}[Proof of Theorem~\ref{thm:nec}]
Let the problem $L(q)$ and the subspectrum $\{ \la_n \}_{n = 1}^{\iy}$ be such that the sequence $\{ v_n \}_{n = 0}^{\iy}$ is incomplete in $\mathcal H$. Then there exists $\hat u \in \mathcal H$, $\hat u \ne 0$, such that
\begin{equation} \label{scaluh}
    (\hat u, v_n)_{\mathcal H} = 0, \quad n \ge 0.
\end{equation}
Since the relations~\eqref{scaluh} are linear by $\hat u$, one can choose $\hat u$ satisfying the estimate $\| \hat u \|_{\mathcal H} \le \eps$ for $\eps$ from Theorem~\ref{thm:loc}. Set $u := [\overline{N(t)}, \overline{K(t)}]$, $\tilde u := u + \hat u = [\overline{\tilde N(t)}, \overline{\tilde K(t)}]$, $\tilde u \ne u$. By Theorem~\ref{thm:loc}, there exists $\tilde q \in L_2(0, \pi)$ such that $\om = \tilde \om$ and $\{ \tilde K, \tilde N \}$ are the Cauchy data of $\tilde q$. 
Define the functions
\begin{align*}
    \tilde \eta_1(\la) & := \frac{s(\pi, \la)}{\la} - \frac{\om c(\pi, \la)}{\la} + \frac{1}{\la} \int_0^{\pi} \tilde K(t) c(t, \la) \, dt, \\ \label{intSp}
    \tilde \eta_2(\la) & := c(\pi, \la) + \frac{\om s(\pi, \la)}{\la} + \frac{1}{\la} \int_0^{\pi} \tilde N(t) s(t, \la) \, dt, \\
    \tilde \Delta(\la) & := f_1(\la) \tilde \eta_2(\la) + f_2(\la) \tilde \eta_1(\la).    
\end{align*}
Clearly, $\tilde \Delta(\la)$ is the characteristic function of $L(\tilde q)$.
The relations~\eqref{scal} and~\eqref{scaluh} yield \eqref{scal2}.
Consequently, the function $\la \tilde \Delta(\la)$ has zeros $\{ \la_n \}_{n \in I}$ of the corresponding multiplicities $\{ m_n \}_{n \in I}$. Thus, $\{ \la_n \}_{n = 1}^{\iy}$ is a subspectrum of $L(\tilde q)$, $\tilde q \ne q$.
\end{proof}

In order to prove Theorem~\ref{thm:cond}, we need several auxiliary lemmas.

\begin{lem} \label{lem:coef}
Suppose that \textsc{(Separation)} is fulfilled.
Then there exist coefficients $\{ C_{n,k} \}$ such that the following relations hold:
\begin{equation} \label{coef}
    \eta_j^{<\nu>}(\la_n) = (-1)^{j - 1} \sum_{k = 0}^{\nu} C_{n,k} f_j^{<\nu - k>}, \quad j = 1, 2, 
\end{equation}
for $n \in I \backslash \{ 0 \}$, $\nu = \overline{0, m_n-1}$ and for $n = 0$, $\nu = \overline{0, m_0 - 2}$. 
\end{lem}

\begin{proof}
Fix $n \in I \backslash \{ 0 \}$.
The relation~\eqref{defD} can be rewritten in the form
\begin{equation} \label{Dfeta}
\eta_1(\la) f_2(\la) + \eta_2(\la) f_1(\la) = \Delta(\la).
\end{equation}
The condition \textsc{(Separation)} and the relation $\Delta(\la_n) = 0$ imply that
$$
[\eta_1(\la_n), \eta_2(\la_n)] = C_{n, 0} [f_1(\la_n), -f_2(\la_n)],
$$
where $C_{n, 0}$ is a nonzero constant, i.e. the relation~\eqref{coef} holds for $\nu = 0$.

Let us prove~\eqref{Dfeta} for $\nu = \overline{1, m_n-1}$ by induction. Assume that~\eqref{coef} is already proved for $\eta_j^{<k>}(\la_n)$, $k = \overline{0, \nu - 1}$, $j = 1, 2$.
Using \eqref{Dfeta} and the relation $\Delta^{<\nu>}(\la_n) = 0$, we get
$$
(\eta_1 f_2)^{<\nu>}(\la_n) = -(\eta_2 f_1)^{<\nu>}(\la_n).
$$
Differentiation of the products yields
$$
\sum_{k = 0}^{\nu} \eta_1^{<k>} f_2^{<\nu - k>} = -\sum_{k = 0}^{\nu} \eta_2^{<k>} f_1^{<\nu - k>}.
$$
Here and below the arguments $(\la_n)$ are omitted for brevity. Using~\eqref{coef} for $\eta_j^{<k>}$, $k = \overline{0, \nu-1}$, we obtain
$$
\eta_1^{<\nu>} f_2 + \sum_{k = 0}^{\nu-1} \sum_{j = 0}^k C_{n, j} f_1^{<k - j>} f_2^{<\nu - k>} = -\eta_2^{<\nu>} f_1 + \sum_{k = 0}^{\nu-1} \sum_{j = 0}^k C_{n, j} f_2^{<k - j>} f_1^{<\nu - k>}.
$$
Calculations show that
\begin{align*}
    \eta_1^{<\nu>} f_2 + \eta_2^{<\nu>} f_1 & = \sum_{j = 0}^{\nu-1} C_{n, j} \sum_{k = j}^{\nu - 1} (f_2^{<k - j>} f_1^{<\nu - k>} - f_1^{<k - j>} f_2^{<\nu - k>})\\  & = \sum_{j = 0}^{\nu - 1} C_{n, j} \left( \sum_{s = 0}^{\nu - j - 1} f_2^{<s>} f_1^{<\nu - j - s>} - \sum_{s = 1}^{\nu - j} f_2^{<s>} f_1^{<\nu - j - s>}\right) \\ & = \sum_{j = 0}^{\nu - 1} C_{n, j} (f_2 f_1^{<\nu - j>} - f_2^{<\nu - j>} f_1).
\end{align*}
Hence
$$
f_2 \left( \eta_1^{<\nu>} - \sum_{j = 0}^{\nu - 1} C_{n, j} f_1^{<\nu - j>}\right) = -f_1 \left( \eta_2^{<\nu>} + \sum_{j = 0}^{\nu - 1} C_{n,j} f_2^{<\nu - j>}\right).
$$
In view of \textsc{(Separation)}, $f_1 \ne 0$ or $f_2 \ne 0$. Consequently, there exists the constant $C_{n, \nu}$ such that
$$
\eta_i^{<\nu>} + (-1)^i \sum_{j = 0}^{\nu - 1} C_{n, j} f_i^{<\nu - j>} = (-1)^i C_{n, \nu} f_i, \quad i = 1, 2.
$$
Thus, the relation~\eqref{coef} is proved for $n \in I \backslash \{ 0 \}$, $\nu = \overline{0,m_n-1}$. Obviously, the arguments above are also valid for $n = 0$, $\nu = \overline{0, m_0-2}$.
\end{proof}

Introduce the vector-functions
\begin{gather} \label{defg}
g(t, \la) := [\eta_1(\la) s(t, \la), -\eta_2(\la) c(t, \la)], \quad
 g_0(t) = [0, 1], \\ \nonumber
 g_{n+\nu}(t) := g^{<\nu>}(t, \la_n), \quad n \in I, \quad \nu = \overline{0, m_n-1}, \quad n + \nu \ge 1.
\end{gather}

\begin{lem} \label{lem:complete}
The sequence $\{ v_n \}_{n = 0}^{\iy}$ is complete in $\mathcal H$ if and only if so does $\{ u_n \}_{n = 0}^{\iy}$. 
\end{lem}

\begin{proof}
Let an element $h = [\overline{h_1}, \overline{h_2}] \in \mathcal H$ be such that 
\begin{equation} \label{scalv}
(h, v_n)_{\mathcal H} = 0, \quad n \ge 0.    
\end{equation} 
The definitions~\eqref{defv}, \eqref{defvn1} and~\eqref{defvn2} imply that
\begin{equation} \label{derV}
    V^{<\nu>}(\la_n) = 0, \quad n \in I, \quad \nu = \overline{0, m_n-1},
\end{equation}
where
$$
V(\la) := \int_0^{\pi} (h_1(t) f_1(\la) s(t, \la) + h_2(t) f_2(\la) c(t, \la)) \, dt.
$$
Obviously,
\begin{equation} \label{binV}
    V^{<\nu>}(\la_n) = \sum_{k = 0}^{\nu} \int_0^{\pi} (h_1(t) f_1^{<k>}(\la_n) s^{<\nu - k>}(t, \la_n) + h_2(t) f_2^{<k>}(\la_n) c^{<\nu - k>}(t, \la_n)) \, dt.
\end{equation}
Consider the function
\begin{equation} \label{defG}
G(\la) := \int_0^{\pi} (h_1(t) \eta_1(\la) s(t, \la) - h_2(t) \eta_2(\la) c(t, \la)) \, dt.
\end{equation}
Let us show that
\begin{equation} \label{derG}
    G^{<\nu>}(\la_n) = 0, \quad n \in I, \quad \nu = \overline{0, m_n-1}.
\end{equation}
Using Lemma~\ref{lem:coef}, \eqref{derV} and~\eqref{binV}, we derive
\begin{align*}
    G^{<\nu>}(\la_n) & = \sum_{k = 0}^{\nu} \int_0^{\pi} (h_1(t) \eta_1^{<k>}(\la_n) s^{<\nu - k>}(t, \la_n) - h_2(t) \eta_2^{<k>}(\la_n) c^{<\nu - k>}(t, \la_n)) \, dt \\
    & = \sum_{k = 0}^{\nu} \sum_{j = 0}^k C_{n,j} \int_0^{\pi} (h_1(t) f_1^{<k - j>}(\la_n) s^{<\nu - k>}(t, \la_n) + h_2(t) f_2^{<k - j>}(\la_n) c^{<\nu - k>}(t, \la_n)) \, dt \\
    & = \sum_{l = 0}^{\nu} C_{n, \nu - l} \sum_{j = 0}^l \int_0^{\pi} (h_1(t) f_1^{<j>}(\la_n) s^{<l - j>}(t, \la_n) + h_2(t) f_2^{<j>}(\la_n) c^{<l - j>}(t, \la_n)) \, dt \\
    & = \sum_{l = 0}^{\nu} C_{n, \nu - l} V^{<l>}(\la_n) = 0, 
\end{align*}
for all $n \in I$, $\nu = \overline{0, m_n-1}$, except for $(n, \nu) = (0, m_0-1)$. Let us consider the case $(n, \nu) = (0, m_0-1)$ separately. Note that
\begin{equation} \label{zero}
\int_0^{\pi} (h_1(t) f_1^{<\nu>}(0) s(t, 0) + h_2(t) f_2^{<\nu>}(0) c(t, 0)) \, dt = f_2^{<\nu>}(0) \int_0^{\pi} h_2(t) \, dt = 0, \quad 
\nu = \overline{0, m_0-1},
\end{equation}
since $s(t, 0) = 0$, $c(t, 0) = 1$, $v_0 = [0, 1]$ and \eqref{scalv} holds for $n = 0$. Combining~\eqref{derV}, \eqref{binV} and~\eqref{zero}, we obtain
$$
V^{<\nu>}(0) = \sum_{k = 0}^{\nu - 1} (h_1(t) f_1^{<k>}(0) s^{<\nu - k>}(t, 0) + h_2(t) f_2^{<k>}(0) c^{<\nu - k>}(t, 0)) \, dt = 0, \quad \nu = \overline{0, m_0-1}.
$$
By using analogous ideas and Lemma~\ref{lem:coef}, we derive
\begin{align*}
    G^{<\nu>}(0) & = \sum_{k = 0}^{\nu - 1} (h_1(t) \eta_1^{<k>}(0) s^{<\nu - k>}(t, 0) - h_2(t) \eta_2^{<k>}(0) c^{<\nu - k>}(t, 0)) \, dt \\ & = \sum_{k = 0}^{\nu - 1} \sum_{j = 0}^k C_{n, j} \int_0^{\pi} (h_1(t) f_1^{<k - j>}(0) s^{<\nu - k>}(t, 0) + h_2(t) f_2^{<k - j>}(0) c^{<\nu - k>}(t, 0)) \, dt \\
    & = \sum_{j = 0}^{\nu - 1} C_{n, j} \sum_{l = 0}^{\nu - j - 1} \int_0^{\pi} (h_1(t) f_1^{<l>}(0) s^{<\nu - l - j>}(t, 0) + h_2(t) f_2^{<l>}(0) c^{<\nu - l - j>}(t, 0)) \, dt \\ & = \sum_{j = 0}^{\nu - 1} C_{n, j} V^{<\nu - j>}(0) = 0, \quad \nu = \overline{0, m_0-1}.
\end{align*}
In particular, \eqref{derG} holds for $n = 0$, $\nu = m_0-1$.

The relations~\eqref{defG} and~\eqref{derG} yield
\begin{equation} \label{scalg}
(h, g_n)_{\mathcal H} = 0, \quad n \ge 0.
\end{equation}
Thus, we have shown that \eqref{scalg} follows from \eqref{scalv}. It can be proved similarly that~\eqref{scalv} follows from~\eqref{scalg}. Therefore the completeness of the sequence $\{ v_n \}_{n = 0}^{\iy}$ is equivalent to the completeness of the sequence $\{ g_n \}_{n = 0}^{\iy}$.
\end{proof}

Further we need two auxiliary propositions. Proposition~\ref{prop:cos} is proved in Appendix of \cite{BK19}.

\begin{prop} \label{prop:cos}
Let $\{ \theta_n \}_{n = 0}^{\iy}$ be a sequence of complex numbers, satisfying the asymptotic formula
\begin{equation} \label{asympttheta}
    \sqrt{\theta_n} = \frac{ \pi n}{a} + \varkappa_n, \quad \{ \varkappa_n \} \in l_2, \quad a > 0.
\end{equation}
Define 
\begin{equation} \label{defmu}
\mu_n := \#\{ k \ge 0 \colon \theta_k = \theta_n \}, \quad J := \{ n \ge 0 \colon \theta_n \ne \theta_k, \, \forall k \colon 0 \le k < n \}.
\end{equation}
Then the sequence $\{ c^{<\nu>}(t, \theta_n) \}_{n \in J, \, \nu = \overline{0, \mu_n-1}}$ is a Riesz basis in $L_2(0, a)$.
\end{prop}

\begin{prop} \label{prop:evenPW}
Let $G(\la)$ be an entire function, satisfying the conditions: 
$$
|G(\rho^2)| \le C \exp(|\mbox{Im} \rho| a), \: \forall \la \in \mathbb C, \qquad \int_{\mathbb R} |G(\rho^2)|^2 d \rho < \iy.
$$
for some positive constants $C$ and $a$. Let $\{ \la_n \}_{n = 0}^{\iy}$ be arbitrary complex numbers, and let the set $I$ and the multiplicities $\{ m_n \}_{n = 0}^{\iy}$ be defined by~\eqref{defI}. Suppose that 
$$
G^{<\nu>}(\la_n) = 0, \quad n \in I, \quad \nu = \overline{0, m_n-1},
$$
and the sequence $\{ c^{<\nu>}(t, \la_n) \}_{n \in I, \, \nu = \overline{0, m_n-1}}$ is complete in $L_2(0, a)$. Then $G(\la) \equiv 0$.
\end{prop}

\begin{proof}
By Paley-Wiener Theorem, the function $G$ can be represented in the form
\begin{equation} \label{intG}
G(\rho^2) = \int_0^{a} r(t) \cos (\rho t) \, dt, \quad r \in L_2(0, 2\pi).
\end{equation}
Differentiating~\eqref{intG}, we get
$$
G^{<\nu>}(\la_n) = \int_0^{a} r(t) c^{<\nu>}(t, \la_n) \, dt = 0, \quad
n \in I, \quad \nu = \overline{0, m_n-1}.
$$
Since the sequence $\{ c^{<\nu>}(t, \la_n) \}_{n \in I, \, \nu = \overline{0, m_n-1}}$ is complete in $L_2(0, a)$, we have $r = 0$ in $L_2(0, a)$, so $G(\la) \equiv 0$.
\end{proof}

\begin{proof}[Proof of Theorem~\ref{thm:cond}(i)]
Suppose that the conditions~\textsc{(Separation)} and~\textsc{(Complete2)} are fulfilled. Let us show that these two conditions imply the completeness of the sequence $\{ g_n \}_{n = 0}^{\iy}$. Let $h \in \mathcal H$ be such that~\eqref{scalg} is valid. We have to show that $h = 0$. Consider the function~$G(\la)$ defined by~\eqref{defG}. The relation~\eqref{scalg} implies \eqref{derG}. In view of~\eqref{intS}, \eqref{intSp},  \eqref{defG}, \eqref{derG} and \textsc{(Complete2)}, the conditions of Proposition~\ref{prop:evenPW} are fulfilled for $a = 2\pi$. Therefore $G(\la) \equiv 0$, i.e.
\begin{equation} \label{G0}
\int_0^{\pi} (h_1(t) \eta_1(\la) s(t, \la) - h_2(t) \eta_2(\la) c(t, \la)) \, dt \equiv 0.
\end{equation}

By virtue of Proposition~\ref{prop:dir}, $\eta_1(\la)$ has a countable set of zeros $\{ \theta_n \}_{n = 1}^{\iy}$, counted with their multiplicities and satisfying the asymptotic formula
\eqref{asympttheta} with $a = \pi$.
Add the value $\theta_0 = 0$. Define the set $J$ and the multiplicities $\{ \mu_n \}_{n \in J}$ by~\eqref{defmu}.
It follows from~\eqref{G0} that
\begin{equation} \label{rel3}
\left( \int_0^{\pi} h_2(t) \eta_2(\la) c(t, \la) \, dt \right)^{<\nu>}_{|\la = \theta_n} = 0, \quad n \in J, \quad \nu = \overline{0, \mu_n-1}.
\end{equation}
Note that $\eta_2(\theta_n) \ne 0$, $n \ge 1$. (Otherwise we have $S(\pi, \theta_n) = S'(\pi, \theta_n) = 0$. Together with equation~\eqref{eqv}, this yields the relation $S(x, \la) \equiv 0$, which is wrong). Consequently, using~\eqref{rel3} and the equality $\int_0^{\pi} h_2(t) \, dt = 0$, we obtain
$$
\int_0^{\pi} h_2(t) c^{<\nu>}(t, \theta_n) = 0, \quad n \in J, \quad \nu = \overline{0, \mu_n-1}.
$$
According to Proposition~\ref{prop:cos}, the sequence $\{ c^{<\nu>}(t, \theta_n) \}_{n \in J, \, \nu = \overline{0, \mu_n - 1}}$ is complete in $L_2(0, \pi)$. Hence $h_2 = 0$ in $L_2(0, \pi)$. Returning to~\eqref{G0}, we easily conclude that also $h_1 = 0$ in $L_2(0, \pi)$.

Thus, we have shown that~\eqref{scalg} implies $h = 0$ in $\mathcal H$, so the sequence $\{ g_n \}_{n = 0}^{\iy}$ is complete in $\mathcal H$. By Lemma~\ref{lem:complete}, the sequence $\{ v_n \}_{n = 0}^{\iy}$ is also complete in $\mathcal H$ under the assumptions of the theorem.
\end{proof}

\begin{lem} \label{lem:gn0}
Let $\{ \tau_n \}_{n \ge 0}$ be arbitrary complex numbers such that $\tau_n \ne \tau_k$ and $\tau_n \ne \overline{\tau_k}$ for all $n \ne k$, $n, k \ge 0$. 
Suppose that the sequence $\{ \cos (\tau_n t) \}_{n = 0}^{\iy}$ is a Riesz basis in $L_2(0, 2\pi)$.
Then the sequence $\{ g_n^0 \}_{n = 0}^{\iy}$ is a Riesz basis in $\mathcal H$, where
\begin{equation} \label{defgn0}
g_n^0(t) := [\sin (\tau_n \pi) \sin (\tau_n t), - \cos(\tau_n \pi) \cos(\tau_n t)].
\end{equation}
\end{lem}

\begin{proof}
In view of \cite[Theorem~3.6.6]{Chr03}, for the sequence $\{ g_n^0 \}_{n = 0}^{\iy}$ to be a Riesz basis in $\mathcal H$, it is sufficient to be complete in $\mathcal H$ and to satisfy the two-side inequality
\begin{equation} \label{twoside}
M_1 \sum_{n = 0}^{N_0} |b_n|^2 \le \left\| \sum_{n = 0}^{N_0} b_n g_n^0 \right\|_{\mathcal H}^2 \le M_2 \sum_{n = 0}^{N_0} |b_n|^2
\end{equation}
for every sequence $\{ b_n \}_{n = 0}^{\iy}$, every integer $N_0 \ge 0$ and some fixed positive constants $M_1$ and $M_2$, independent of $\{ b_n \}$ and $N_0$.

First, we show that the sequence $\{ g_n^0 \}_{n = 0}^{\iy}$ is complete in $\mathcal H$. Let $h = [\overline{h_1}, \overline{h_2}] \in \mathcal H$ be such that $(h, g_n^0)_{\mathcal H} = 0$ for all $n \ge 0$. It means that the function
$$
G_0(\la) := \int_0^{\pi} (h_1(t) \sin (\sqrt{\la} \pi) \sin (\sqrt{\la} t) - h_2(t) \cos (\sqrt{\la} \pi) \cos (\sqrt{\la} t)) \, dt
$$
has zeros $\{ \tau_n^2 \}_{n \ge 0}$. Applying Proposition~\ref{prop:evenPW}, we conclude that $G_0(\la) \equiv 0$. Then one can easily show that $h_1 = h_2 = 0$ in $L_2(0, \pi)$, so $\{ g_n^0 \}_{n = 0}^{\iy}$ is complete.

Second, we prove the two-side inequality~\eqref{twoside}. Calculations show that
\begin{align*}
(g_n^0, g_k^0)_{\mathcal H} & = \int_0^{\pi} (\sin (\overline{\tau_n} \pi) \sin (\overline{\tau_n} t) \sin (\tau_k \pi) \sin(\tau_k t) + \cos (\overline{\tau_n} \pi) \cos (\overline{\tau_n} t) \cos (\tau_k \pi) \cos(\tau_k t)) \, dt \\ & = \frac{\sin (2 (\overline{\tau_n} - \tau_k) \pi)}{2 (\overline{\tau_n} - \tau_k)} = \int_0^{2\pi} \cos (\overline{\tau_n} t) \cos(\tau_k t) \, dt.
\end{align*}
Hence
$$
\left\| \sum_{n = 0}^{N_0} b_n g_n^0 \right\|_{\mathcal H} = 
\sum_{n = 0}^{N_0} \sum_{k = 0}^{N_0} \overline{b_n} b_k (g_n^0, g_k^0)_{\mathcal H} = \left\| \sum_{n = 0}^{N_0} b_n \cos (\tau_n t) \right\|_{L_2(0, 2\pi)}.
$$
Since the sequence $\{ \cos (\tau_n t) \}_{n = 0}^{\iy}$ is a Riesz basis in $L_2(0, 2\pi)$, the two-side inequality similar to~\eqref{twoside} is valid for this sequence. Consequently, the inequality~\eqref{twoside} is also valid for $\{ g_n^0 \}_{n = 0}^{\iy}$, so $\{ g_n^0 \}_{n = 0}^{\iy}$ is a Riesz basis in $\mathcal H$.
\end{proof}

\begin{proof}[Proof of Theorem~\ref{thm:cond}(ii)]
Suppose that the conditions \textsc{(Separation)}, \textsc{(Simple)}, \textsc{(Asymptotics)} and~\textsc{(Basis2)} are fulfilled. First, let us show that $\{ g_n \}_{n = 0}^{\iy}$ is a Riesz basis in $\mathcal H$. Since \textsc{(Basis2)} implies \textsc{(Complete2)}, the conditions of Theorem~\ref{thm:cond}(i) hold, so the sequence $\{ g_n \}_{n = 0}^{\iy}$ is complete in $\mathcal H$ according to the previous proof. Substituting~\eqref{intS} and~\eqref{intSp} into~\eqref{defg}, we get
$$
g(t, \rho^2) = [\sin (\rho \pi) \sin (\rho t), -\cos (\rho \pi) \cos (\rho t)] + O\left( \rho^{-1} \exp(2|\mbox{Im}\,\rho|\pi)\right), \quad |\rho| \to \iy.
$$
Substituting $\rho = \rho_n$ into the latter relation and taking the conditions \textsc{(Simple)} and \textsc{(Asymptotics)} into account, we conclude that
$\{ \| g_n - g_n^0 \|_{\mathcal H} \}_{n \ge 0} \in l_2$, where $g_n^0$ is defined by~\eqref{defgn0} for $n \ge 0$.
Here $\tau_n = \rho_n$ for $n \ge n_0$ and $\{ \tau_n \}_{n = 0}^{n_0 - 1}$ are arbitrary complex numbers, such that $\tau_n \ne \tau_k$ and $\tau_n \ne \overline{\tau_k}$ for all $n \ne k$, $n, k \ge 0$. Thus, the sequence $\{ \tau_n \}_{n \ge 0}$ satisfies the conditions of Lemma~\ref{lem:gn0}. The Riesz-basis property of the sequence $\{ \cos(\tau_n t) \}_{n = 0}^{\iy}$ follows from \textsc{(Basis2)}. Applying Lemma~\ref{lem:gn0}, we conclude that $\{ g_n^0 \}_{n = 0}^{\iy}$ is a Riesz basis in $\mathcal H$. Thus, the sequence $\{ g_n \}_{n = 0}^{\iy}$ is complete and $l_2$-close to the Riesz basis $\{ g_n^0 \}_{n = 0}^{\iy}$ in $\mathcal H$. Hence $\{ g_n \}_{n = 0}^{\iy}$ is also a Riesz basis.

Second, let us show that the Riesz-basis property of $\{ g_n \}_{n = 0}^{\iy}$ implies that $\{ v_n \}_{n = 0}^{\iy}$ is an unconditional basis in $\mathcal H$, i.e. the normalized sequence $\{ v_n / \| v_n \|_{\mathcal H} \}_{n = 0}^{\iy}$ is a Riesz basis. By Lemma~\ref{lem:complete}, the sequence $\{ v_n \}_{n = 0}^{\iy}$ is complete in $\mathcal H$. Recall that, by \textsc{(Simple)}, the eigenvalues $\{ \la_n \}$ are simple for sufficiently large $n$. Therefore, by Lemma~\ref{lem:coef}, we have
$g_n = k_n v_n$, $n \ge n_0$, where $\{ k_n \}_{n \ge n_0}$ are nonzero constants. This fact together with the completeness of $\{ v_n \}_{n = 0}^{\iy}$ yield the claim.
\end{proof}

\section{Hochstadt-Lieberman problem}

In this section, we show one of the applications of our main results.
Consider the following eigenvalue problem:
\begin{gather} \label{eqv2}
    -y''(x) + q(x) y(x) = \la y(x), \quad x \in (0, 2\pi), \\ \label{dbc}
    y(0) = y(2\pi) = 0,
\end{gather}
with a complex-valued potential $q \in L_2(0, 2 \pi)$. Denote by $\{ \la_n \}_{n = 1}^{\iy}$ the eigenvalues of the problem~\eqref{eqv2}-\eqref{dbc}, counted with their multiplicities and numbered according to their asymptotics
\begin{equation} \label{asymptla}
\sqrt{\la_n} = \frac{n}{2} + \frac{\Omega}{\pi n} + o\left( n^{-1}\right), \quad n \to \iy, 
\end{equation}
where $\Omega := \frac{1}{2} \int_0^{2\pi} q(x) \, dx$.

The Hochstadt-Lieberman problem, also called the half-inverse problem, is formulated as follows.

\begin{ip} \label{ip:HL}
Suppose that the potential $q(x)$ is known a priori for $x \in (\pi, 2\pi)$. Given the spectrum $\{ \la_n \}_{n = 1}^{\iy}$, find the potential $q(x)$ for $x \in (0, \pi)$.
\end{ip}

Inverse Problem~\ref{ip:HL} and its generalizations were studied in \cite{HL78, GS00, Hor01, HM04, MP10, But11} and other papers. In this section, we show that this problem can be treated as a special case of Inverse Problem~\ref{ip:main}.

Denote by $S(x, \la)$ and $\psi(x, \la)$ the solutions of equation~\eqref{eqv2}, satisfying the initial conditions $S(0, \la) = 0$, $S'(0, \la) = 1$, $\psi(2\pi, \la) = 0$, $\psi'(2\pi, \la) = -1$. The eigenvalues of the problem~\eqref{eqv2}-\eqref{dbc} coincide with the zeros of the characteristic function
\begin{equation} \label{defD2}
\Delta(\la) = \psi(\pi, \la) S'(\pi, \la) - \psi'(\pi, \la) S(\pi, \la).
\end{equation}
Comparing~\eqref{defD2} with~\eqref{defD}, we conclude that the eigenvalue problem~\eqref{eqv2}-\eqref{dbc} is equivalent to~\eqref{eqv}-\eqref{bc} with 
\begin{equation} \label{fpsi}
f_1(\la) := \psi(\pi, \la), \quad f_2(\la) := -\psi'(\pi, \la). 
\end{equation}
Note that these functions $f_j(\la)$, $j = 1, 2$, are entire in $\la$-plane, and they can be constructed by the known part of the potential $q(x)$, $x \in (\pi, 2\pi)$.
The constant $\om$ also can be easily determined by the given data of Inverse Problem~\ref{ip:HL}. Indeed, we have
$$
\om = \frac{1}{2} \int_0^{\pi} q(x) \, dx = \Omega - \frac{1}{2} \int_{\pi}^{2\pi} q(x) \, dx,
$$
and the constant $\Omega$ can be found from the eigenvalue asymptotics~\eqref{asymptla}.
Thus, Inverse Problem~\ref{ip:HL} is reduced to Inverse Problem~\ref{ip:main} by the whole spectrum $\{ \la_n \}_{n = 1}^{\iy}$ of~\eqref{eqv2}-\eqref{dbc}.

\begin{prop}
Let $f_j(\la)$, $j = 1, 2$, be entire functions defined by~\eqref{fpsi}, and let $\{ \la_n \}_{n = 1}^{\iy}$ be the eigenvalues of the problem~\eqref{eqv2}-\eqref{dbc} counted with their multiplicities, $\la_0 := 0$. Then the conditions \textsc{(Basis2)}, \textsc{(Separation)}, \textsc{(Simple)} and~\textsc{(Asymptotics)} are fulfilled.
\end{prop}

\begin{proof}
The condition~\textsc{(Basis2)} follows from the asymptotics~\eqref{asymptla} and Proposition~\ref{prop:cos}. \textsc{(Separation)} is fulfilled, because the functions $\psi(\pi, \la)$ and $\psi'(\pi, \la)$ do not have common zeros. Indeed, if $\psi(\pi, \mu) = \psi'(\pi, \mu) = 0$ for some $\mu \in \mathbb C$, then $\psi(x, \mu)$ is the solution of the initial value problem for equation~\eqref{eqv2} with the zero conditions at $x = \pi$. Then $\psi(x, \mu) \equiv 0$, which is impossible. The conditions \textsc{(Simple)} and \textsc{(Asymptotics)} easily follow from the asymptotics~\eqref{asymptla}.
\end{proof}

Thus, our main results can be applied to the Hochstadt-Lieberman problem. In particular, Theorem~\ref{thm:uniq} implies the following corollary, which generalizes the Hochstadt-Lieberman uniqueness theorem \cite{HL78} to the case of complex-valued potentials.

\begin{thm}
Let $\{ \la_n \}_{n = 1}^{\iy}$ and $\{ \tilde \la_n \}_{n = 1}^{\iy}$ be the spectra of the boundary value problems in the form~\eqref{eqv2}-\eqref{dbc} for potentials $q$ and $\tilde q$, respectively. Suppose that $q(x) = \tilde q(x)$ a.e. on $(\pi, 2\pi)$ and $\la_n = \tilde \la_n$ for all $n \ge 1$. Then $q(x) = \tilde q(x)$ a.e. on $(0, \pi)$. In other words, the solution of Inverse Problem~\eqref{ip:HL} is unique.
\end{thm}

Algorithm~\ref{alg:ip} can be used for constructive solution of Inverse Problem~\ref{ip:main}. This algorithm generalizes the methods, developed in parallel by Martinyuk and Pivovarchik \cite{MP10} and by Buterin \cite{But11} for solving the Hochstadt-Lieberman problem.

\section*{Appendix. Inverse problem by Cauchy data}

\setcounter{section}{5}
\setcounter{equation}{0}
\setcounter{thm}{0}

The goal of this section is to prove the following theorem on local solvability and stability of Inverse Problem~\ref{ip:cd}.

\begin{thm} \label{thm:loc}
Let $q$ be a fixed complex-valued function from $L_2(0, \pi)$, and let $\{ K, N \}$ be the corresponding Cauchy data. Then there exists $\eps > 0$ (depending on $q$) such that, for any functions $\tilde K, \tilde N$ from $L_2(0, \pi)$ satisfying the estimate
\begin{equation} \label{estNK}
    \Xi := \max \{ \| K - \tilde K \|_{L_2(0, \pi)}, \| N - \tilde N \|_{L_2(0, \pi)} \} \le \eps,
\end{equation}
there exists a unique function $\tilde q \in L_2(0, \pi)$ such that $\int_0^{\pi} (q(x) - \tilde q(x)) \, dx = 0$ and $\{ \tilde K, \tilde N \}$ are the Cauchy data for $\tilde q$. In addition,
\begin{equation} \label{estq}
    \| q - \tilde q \|_{L_2(0, \pi)} \le C \Xi,
\end{equation}
where the constant $C$ depends only on $q$ and not on $\{ \tilde K, \tilde N \}$.
\end{thm}

Below the symbol $C$ is used for various positive constants. In order to prove Theorem~\ref{thm:loc}, we need several auxiliary propositions.
Applying the standard approach (see, e.g., \cite[Theorem~1.1.3]{FY01}), based on Rouch\'{e}'s Theorem, one can easily obtain the following result.

\begin{prop} \label{prop:dir}
Let $K(t)$ be an arbitrary complex-valued function from $L_2(0, \pi)$. Then the function $\eta_1(\la)$ defined by~\eqref{intS} has the countable set of zeros $\{ \theta_n \}_{n = 1}^{\iy}$ numbered according to their multiplicities so that $|\theta_n| \le |\theta_{n+1}|$, $n \in \mathbb N$, and satisfying the asymptotic formula
\begin{equation} \label{asymptnu}
    \nu_n := \sqrt{\theta_n} = n + O\left( n^{-1} \right), \quad n \in \mathbb N.
\end{equation}
\end{prop}

In view of the asymptotic formula \eqref{asymptnu}, we can find the smallest integer $n_1 \ge 2$ such that the zeros $\{ \theta_n \}$ are simple for $n \ge n_1$ and $|\theta_{n_1}| > |\theta_{n_1 - 1}|$. Define the contour $\ga_0 := \{ \la \in \mathbb C \colon |\la| = (|\theta_{n_1}| + |\theta_{n_1-1}|)/2 \}$. Clearly, $\theta_n \in \mbox{int}\, \ga_0$ for $n = \overline{1, n_1-1}$ and the eigenvalues $\{ \theta_n \}_{n = n_1}^{\iy}$ lie strictly outside $\ga_0$.

Without loss of generality, we may assume that equal eigenvalues in the sequence $\{ \theta_n \}_{n = 1}^{\iy}$ are consecutive. Introduce the notations
\begin{gather*}
\mathcal S := \{ 1 \} \cup \{ n \ge 2 \colon \la_n \ne \la_{n-1} \}, \quad k_n := \# \{ k \in \mathbb N \colon \theta_k = \theta_n \}, \\
M(\la) := \frac{\eta_2(\la)}{\eta_1(\la)}, \quad M_{n + \nu} := \Res_{\la = \theta_n} (\la - \theta_n)^{\nu} M(\la), \quad n \in \mathcal S, \quad \nu = \overline{0, k_n-1}.
\end{gather*}
Below we agree that, if a certain symbol $\ga$ denotes an object constructed by $\{ K, N, \om \}$, then the symbol $\tilde \ga$ with tilde denotes the analogous object constructed by $\{ \tilde K, \tilde N, \om\}$.

\begin{lem} \label{lem:KN}
Let $K$, $N$ be fixed complex-valued functions from $L_2(0, \pi)$, and let $\om \in \mathbb C$. Then there exists $\eps > 0$ (depending on $K$, $N$, $\om$) such that, for any $\tilde K, \tilde N \in L_2(0, \pi)$ satisfying~\eqref{estNK}, the points $\{ \tilde \theta_n \}_{n = 1}^{n_1-1}$ lie strictly inside $\ga_0$ and
\begin{equation} \label{estM}
    \max_{\la \in \ga_0} |M(\la) - \tilde M(\la)| \le C \Xi.
\end{equation}
For $n \ge n_1$, we have $\tilde k_n = 1$ and
\begin{equation} \label{estxi}
    \left( \sum_{n = n_1}^{\iy} (n \xi_n)^2 \right)^{1/2} \le C \Xi,
\end{equation}
where $\xi_n := |\nu_n - \tilde \nu_n| + \frac{1}{n^2} |M_n - \tilde M_n|$. The constant $C$ in the estimates~\eqref{estM} and~\eqref{estxi} depends only on $\{ K, N, \om \}$.
\end{lem}

\begin{proof} \textsf{Step 1.} If the functions $K$, $N$, $\tilde K$, $\tilde N$ and the number $\om$ satisfy the conditions of the lemma for sufficiently small $\eps > 0$, then~\eqref{intS} and~\eqref{estNK} yield the estimates
\begin{gather} \label{eta1below}
    |\eta_1(\la)|, |\tilde \eta_1(\la)| \ge c_0 > 0, \quad \la \in \ga_0, \\ \nonumber
    |\eta_1(\la) - \tilde \eta_1(\la)| \le C \Xi, \quad \la \in \ga_0.
\end{gather}
Consequently, for sufficiently small $\eps > 0$, we have $\dfrac{|\eta_1(\la) - \tilde \eta_1(\la)|}{|\eta_1(\la)|} < 1$ on $\ga_0$. By Rouch\'{e}'s Theorem, the function $\tilde \eta_1(\la)$ has inside $\ga_0$ the same number of zeros as $\eta_1(\la)$. According to our notations, these zeros of $\tilde \eta_1(\la)$ are $\{ \tilde \theta_n \}_{n = 1}^{n_1 - 1}$.

Using~\eqref{intS}, \eqref{estNK} and~\eqref{eta1below}, we obtain the estimate~\eqref{estM}:
$$
|M(\la) - \tilde M(\la)| = \frac{|\eta_2(\la) \tilde \eta_1(\la) - \tilde \eta_2(\la) \eta_1(\la)|}{|\eta_1(\la)| |\tilde \eta_1(\la)|} \le C \Xi, \quad \la \in \ga_0.
$$

\textsf{Step 2.} For $n \ge n_1$, consider the contours $\ga_{n, r} := \{ \rho \in \mathbb C \colon |\rho - \nu_n| = r \}$, where $r > 0$ is fixed and so small that $r \le \frac{|\nu_n - \nu_{n+1}|}{2}$, $n \ge n_1$. The function $\eta_0(\rho^2)$ has exactly one zero $\nu_n \in \mbox{int}\, \ga_{n, r}$ in $\rho$-plane for every $n \ge n_1$. The relations~\eqref{intS} and~\eqref{estNK} yield the estimate 
\begin{equation} \label{sm1}
|\eta_1(\rho^2)| \ge \frac{c_r}{n}, \quad \rho \in \ga_{n, r}, \quad n \ge n_1,
\end{equation}
where the constant $c_r$ depends on $r$ and not on $\rho$ and $n$. For sufficiently small $\eps > 0$, we obtain the estimate
\begin{equation} \label{sm2}
|\tilde \eta_1(\rho^2) - \eta_1(\rho^2)| \le \frac{C \Xi}{n^2}, \quad \rho \in \ga_{n, r}, \quad n \ge n_1.
\end{equation}
Using \eqref{sm1}, \eqref{sm2} and applying Rouch\'{e}'s Theorem to the contour $\ga_{n,r}$ in $\rho$-plane, we conclude that $\tilde \eta_1(\rho^2)$ has exactly one zero $\tilde \nu_n \in \mbox{int}\, \ga_{n, r}$ for each $n \ge n_1$.

Using the Taylor formula
$$
\eta_1(\tilde \nu_n^2) = \eta_1(\nu_n^2) + \frac{d}{d\rho} \eta_1(\rho^2)_{|\rho = \zeta_n} (\tilde \nu_n - \nu_n), \quad \zeta_n \in \mbox{int} \, \ga_{n, r},
$$
and~\eqref{intS}, we derive the relation
\begin{equation} \label{sm3}
\eta_1(\tilde \nu_n^2) - \tilde \eta_1(\tilde \nu_n^2) = \frac{1}{\tilde \nu_n^2} \int_0^{\pi} \hat K(t) \cos (\tilde \nu_n t) \, dt = \frac{d}{d\rho} \eta_1(\rho^2)_{|\rho = \zeta_n} (\tilde \nu_n - \nu_n),
\end{equation}
where $\hat K := K - \tilde K$. It is easy to check that
\begin{equation} \label{sm4}
    \left| \frac{d}{d\rho}\eta_1(\rho^2)\right| \ge \frac{C}{n}, \quad \rho \in \mbox{int}\, \ga_{n, r}, \quad n \ge n_1.
\end{equation}
Using \eqref{asymptnu} and \eqref{estNK}, we estimate the integral:
\begin{align} \nonumber
\left| \int_0^{\pi} \hat K(t) \cos (\tilde \nu_n t) \, dt\right| & \le \left| \int_0^{\pi} \hat K(t) \cos (nt) \, dt\right| + \left| \int_0^{\pi} \hat K(t) (\cos (\tilde \nu_n t) - \cos (n t)) \, dt\right| \\ \label{sm5} & \le |\hat K_n| + \frac{C \Xi}{n}, \quad n \ge n_1, \quad \hat K_n := \int_0^{\pi} \hat K(t) \cos (nt) \, dt.
\end{align}
Combining~\eqref{sm3}, \eqref{sm4} and~\eqref{sm5}, we obtain
\begin{equation} \label{sm6}
|\tilde \nu_n - \nu_n| \le \frac{C |\hat K_n|}{n} + \frac{C \Xi}{n^2}, \quad n \ge n_1.
\end{equation}
Bessel's inequality for the Fourier coefficients $\{ \hat K_n \}$ and~\eqref{estNK} imply that
\begin{equation} \label{bessel}
\left( \sum_{n = n_1}^{\iy} |\hat K_n|^2 \right)^{1/2} \le C \Xi.
\end{equation}
Combining~\eqref{sm6} and~\eqref{bessel}, we arrive at the estimate
\begin{equation} \label{sumnu}
\left( \sum_{n = n_1}^{\iy}n^2 |\tilde \nu_n - \nu_n|^2 \right)^{1/2} \le C \Xi.
\end{equation}

\textbf{Step 3.} Note that $\{ \theta_n \}_{n = n_1}^{\iy}$ are simple poles of $M(\la)$, so
$$
M_n = \Res_{\la = \theta_n} M(\la) = \frac{\eta_2(\theta_n)}{\dot \eta_1(\theta_n)}, \quad n \ge n_1,
$$
where $\dot f(\la) = \frac{d}{d\la} f(\la)$. 
If $\eps > 0$ is sufficiently small, the analogous relation is valid for $\tilde M_n$, $n \ge n_1$.
Hence
\begin{equation} \label{subMn}
\tilde M_n - M_n = \frac{(\tilde \eta_2 - \eta_2) \dot \eta_1 + \eta_2 (\dot \eta_1 - \dot {\tilde \eta}_1)}{\dot \eta_1 \dot{\tilde \eta}_1}_{|\la = \theta_n}, \quad n \ge n_1.
\end{equation}
Using~\eqref{subMn} and the following estimates
\begin{gather*}
    |\dot \eta_1(\theta_n)| \ge \frac{C}{n^2}, \quad 
    |\dot {\tilde\eta}_1(\theta_n)| \ge \frac{C}{n^2}, \quad |\dot \eta_1(\theta_n)| \le \frac{C}{n^2}, \quad |\eta_2(\theta_n)| \le C, \\
     \quad |\tilde \eta_2(\theta_n) - \eta_2(\theta_n)| \le \frac{C |\hat N_n|}{n} + \frac{C \Xi}{n^2}, \quad 
     |\dot \eta_1(\theta_n) - \dot {\tilde \eta}_1(\theta_n)| \le \frac{C |\hat L_n|}{n^3} + \frac{C \Xi}{n^4}, \quad n \ge n_1,
\end{gather*}
where 
$$
\hat N_n := \int_0^{\pi} \hat N(t) \sin (n t) \, dt, \quad \hat N := N - \tilde N, \quad
\hat L_n := \int_0^{\pi} t \hat K(t) \sin (n t) \, dt,
$$
we arrive at the estimate
$$
|\tilde M_n - M_n| \le C n (|\hat N_n| + |\hat L_n|) + C \Xi, \quad n \ge n_1.
$$
Similarly to~\eqref{sumnu}, we obtain
\begin{equation} \label{sumMn}
\left( \sum_{n = n_1}^{\iy} n^{-2} |\tilde M_n - M_n|^2 \right)^{1/2} \le C \Xi.
\end{equation}
The relations~\eqref{sumnu} and~\eqref{sumMn} together imply \eqref{estxi}.
\end{proof}

In \cite{BSY13, Bond20} the following inverse problem has been studied.

\begin{ip} \label{ip:gsd}
Given the data $\{ \theta_n, M_n \}_{n = 1}^{\iy}$, find $q$.
\end{ip}

Clearly, Inverse Problem~\ref{ip:gsd} is equivalent to Inverse Problem~\ref{ip:cd} by the Cauchy data. In addition, one can uniquely construct $M(\la)$ by $\{ \theta_n, M_n \}_{n = 1}^{\iy}$ and vice versa. In \cite{Bond20}, the following proposition on local solvability and stability of Inverse Problem~\ref{ip:gsd} has been proved.

\begin{prop} \label{prop:gsd}
Let $q \in L_2(0, \pi)$ be fixed. Then there exists $\eps > 0$ (depending on $q$) such that, for any complex numbers $\{ \tilde \theta_n, \tilde M_n \}_{n = 1}^{\iy}$ satisfying the estimate
$$
\Omega := \max \left\{ \max_{\la \in \ga_0} |M(\la) - \tilde M(\la)|, \biggl( \sum_{n = n_1}^{\iy} (n \xi_n)^2 \biggr)^{1/2} \right\} \le \eps,
$$
there exists the unique complex-valued function $\tilde q \in L_2(0, \pi)$ being the solution of Inverse Problem~\ref{ip:gsd} for $\{ \tilde \theta_n, \tilde M_n \}_{n = 1}^{\iy}$. Moreover, the estimate~\eqref{estq} holds with the constant $C$ depending only on $q$.
\end{prop}

Lemma~\ref{lem:KN} and Proposition~\ref{prop:gsd} together imply Theorem~\ref{thm:loc}.

\medskip

{\bf Acknowledgments.} This work was supported by Grant 20-31-70005 of the Russian Foundation for Basic Research.

\noindent Natalia Pavlovna Bondarenko \\
1. Department of Applied Mathematics and Physics, Samara National Research University, \\
Moskovskoye Shosse 34, Samara 443086, Russia, \\
2. Department of Mechanics and Mathematics, Saratov State University, \\
Astrakhanskaya 83, Saratov 410012, Russia, \\
e-mail: {\it BondarenkoNP@info.sgu.ru}


\begin{thebibliography}{99}
\bibitem{Mar77}
Marchenko, V.A. Sturm-Liouville Operators and Their Applications, Naukova Dumka, Kiev (1977) (Russian); English transl., Birkhauser (1986).

\bibitem{Lev84}
Levitan, B.M. Inverse Sturm-Liouville Problems, Nauka, Moscow (1984) (Russian); English transl., VNU Sci. Press, Utrecht (1987).

\bibitem{PT87}
P\"{o}schel, J.; Trubowitz, E. Inverse Spectral Theory, New York, Academic Press (1987).

\bibitem{FY01}
Freiling, G.; Yurko, V. Inverse Sturm-Liouville Problems and Their Applications, Huntington, NY: Nova Science Publishers (2001).

\bibitem{BS97}
Browne, P.J.; Sleeman, B.D. A uniqueness theorem for inverse eigenparameter
dependent Sturm-Liouville problems. Inverse Problems 13 (1997), no.6, 1453--1462.

\bibitem{Yur00}
Yurko, V. A. An inverse problem for pencils of differential operators,
Matem. Sbornik, 191 (2000), 137--160. 

\bibitem{Gul05}
Guliyev, N.J. Inverse eigenvalue problems for Sturm-Liouville equations with
spectral parameter linearly contained in one of the boundary condition, Inverse
Problems 21 (2005), no.4, 1315--1330.

\bibitem{YH10}
Yang, C.-F.; Huang, Z.-Y. A half-inverse problem with eigenparameter dependent boundary
conditions, Numerical Functional Analysis and Optimization 31 (2010), no.~6, 754--762.

\bibitem{BBW04}
Binding, P. A.; Browne, P. J.; Watson, B. A. Equivalence of inverse Sturm-Liouville problems with boundary conditions rationally dependent on the eigenparameter, J. Math. Anal. Appl. 291 (2004), 246--261.

\bibitem{Chug01}
Chugunova, M.V. Inverse spectral problem for the Sturm-Liouville operator with
eigenvalue parameter dependent boundary conditions. Oper. Theory: Advan. Appl. 123, Birkhauser, Basel (2001), 187--194.

\bibitem{FY10}
Freiling, G.; Yurko V.A. Inverse problems for Sturm-Liouville equations with boundary conditions polynomially dependent on the
spectral parameter, Inverse Problems 26 (2010), 055003 (17pp).

\bibitem{FY12}
Freiling, G.; Yurko, V. Determination of singular differential pencils from the Weyl function, Advances in Dynamical Systems and
Applications 7 (2012), no. 2, 171--193.

\bibitem{YX15}
Yang, C-F.; Xu, X.-C. Ambarzumyan-type theorem with polynomially dependent eigenparameter,
Math. Meth. Appl. Sci. 38 (2015), 4411--4415.

\bibitem{Gul19}
Guliyev, N.J. Schr\"odinger operators with distributional potentials and boundary conditions
dependent on the eigenvalue parameter, J. Math. Phys. 60 (2019), 063501.

\bibitem{HL78}
Hochstadt, H.; Lieberman, B. An inverse Sturm-Liouville problem with mixed given data, SIAM J. Appl. Math. 34 (1978), no.~4, 676--680.

\bibitem{Hald84}
Hald O. Discontinuous inverse eigenvalue problem, Commun. Pure Appl. Math. 37 (1984), 53--577.

\bibitem{SY08}
Shieh, C.-T.; Yurko, V.A. Inverse nodal and inverse spectral problems for discontinuous boundary value problems,
J. Math. Anal. Appl. 347 (2008), 266--272.

\bibitem{Yang14}
Yang, C.-F. Inverse problems for the Sturm-Liouville operator with discontinuity, Inverse Problems in Science
and Engineering 22 (2014), no.~2, 232--244.

\bibitem{Wang15}
Wang, Y.P. Inverse problems for discontinuous Sturm-Liouville operators with mixed spectral data, Inverse Problems in Science and Engineering 23 (2015), no.~7, 1180--1198.

\bibitem{YB19}
Yang, C.-F.; Bondarenko, N.P. Local solvability and stability of inverse problems for Sturm-Liouville operators with a discontinuity, Journal of Differential Equations (2019), published online, https://doi.org/10.1016/j.jde.2019.11.035

\bibitem{Piv00}
Pivovarchik, V.N. Inverse problem for the Sturm-Liouville equation on a simple graph,
SIAM J. Math. Anal. 32 (2000), no.~4, 801--819.

\bibitem{Yang10}
Yang, C.-F. Inverse spectral problems for the Sturm-Liouville operator on a $d$-star graph, 
J. Math. Anal. Appl. 365 (2010), 742--749.

\bibitem{YW17}
Yang, C.-F.; Wang, F. Inverse problems on graphs with loops, J. Inverse Ill-Posed Probl. 25 (2017), no.~3, 373--380.

\bibitem{Bond18}
Bondarenko, N.P. A partial inverse problem for the Sturm-Liouville operator on a star-shaped graph, Anal. Math. Phys. 8 (2018), no.~1, 155--168.

\bibitem{Bond-tamk}
Bondarenko, N.P. A 2-edge partial inverse problem for the Sturm-Liouville operators with singular potentials on a star-shaped graph, Tamkang J. Math. 49 (2018), no. 1, 49-66.

\bibitem{BCFK06}
Berkolaiko, G.; Carlson, R.; Fulling, S.; Kuchment, P. Quantum Graphs and Their Applications,
Contemp. Math. 415, Amer. Math. Soc., Providence, RI (2006).

\bibitem{MP94} 
McLaughlin, J.R.; Polyakov, P.L. On the uniqueness of a spherically symmetric speed of sound
from transmission eigenvalues, J. Diff. Eqns. 107 (1994), 351--382.

\bibitem{MPS94}
McLaughlin, J.R.; Polyakov, P.L.; Sacks, P.E., Reconstruction of a spherically symmetric speed of sound, SIAM J. Appl. Math. 54 (1994), 1203---1223.

\bibitem{MSS97}
McLaughlin, J.R.; Sacks, P.E.; Somasundaram, M., 
Inverse scattering in acoustic media using interior transmission eigenvalues, in: G. Chavent, G. Papanicolaou, P. Sacks, W. Symes (Eds.), Inverse Problems in Wave Propagation, Springer, New York (1997), 357--374.

\bibitem{BB17}
Bondarenko, N.; Buterin, S. On a local solvability and stability of the inverse transmission eigenvalue problem, 
Inverse Problems 33 (2017), 115010. 

\bibitem{Bond-future}
Bondarenko, N. Solvability and stability of the inverse Sturm-Liouville problem with analytical functions in the boundary condition (to appear).

\bibitem{But07}
Buterin, S.A. On inverse spectral problem for non-selfadjoint Sturm-Liouville operator on a finite interval, J. Math. Anal. Appl. 335 (2007), no. 1, 739--749.

\bibitem{BSY13}
Buterin, S.A.; Shieh, C.-T.; Yurko, V.A. Inverse spectral problems for non-selfadjoint second-order differential operators with Dirichlet boundary conditions, Boundary Value Problems (2013), 2013:180.

\bibitem{BK19}
Buterin, S.; Kuznetsova, M. On Borg's method for non-selfadjoint Sturm-Liouville operators, Anal. Math. Phys. 9 (2019), 2133--2150. 

\bibitem{Chr03}
Christensen, O. An Introduction to Frames and Riesz Bases, Applied and Numerical Harmonic Analysis, Birkhauser, Boston (2003).

\bibitem{GS00}
Gesztesy, F.; Simon, B. Inverse spectral analysis with partial information on the potential, II. The case of discrete spectrum,
Trans. AMS 352 (2000), no.~6, 2765--2787.

\bibitem{Hor01}
Horvath, M. On the inverse spectral theory of Schr\"odinger and Dirac operators, Trans. AMS 353 (2001), no. 10, 4155--4171.

\bibitem{HM04}
Hryniv, R.O.; Mykytyuk, Ya.V.  Half-inverse spectral problems for Sturm-Liouville
operators with singular potentials, Inverse Problems 20 (2004), 1423--1444.

\bibitem{MP10}
Martinyuk, O.; Pivovarchik, V. On the Hochstadt-Lieberman theorem, Inverse Problems 26 (2010), 035011 (6pp).

\bibitem{But11}
Buterin, S.A. On half inverse problem for differential pencils with the spectral parameter in boundary conditions,
Tamkang J. Math. 42 (2011), 355--364.

\bibitem{Bond20}
Bondarenko, N.P. Local solvability and stability of the inverse problem for the non-self-adjoint Sturm-Liouville operator, preprint (2020), arXiv:2002.05045 [math.SP]. 

\bibitem{BFY14}
Buterin, S. A.; Freiling, G.; Yurko, V. A. Lectures in the theory of entire functions, Schriftenriehe der Fakult\"at f\"ur Matematik, Duisbug-Essen University, SM-UDE-779 (2014).

\end{thebibliography}
\end{document}